\documentclass[12pt,reqno]{amsart}

\usepackage{times}
\usepackage{amsmath,amssymb,mathrsfs}

\newtheorem{thm}{Theorem}
\newtheorem{lemma}[thm]{Lemma}

\newtheorem*{mthm}{Main Theorem}

\newcommand{\C}{{\mathbb C}}
\newcommand{\R}{{\mathbb R}}

\newcommand{\T}{{\mathbb T}}
\newcommand{\D}{{\mathbb D}}

\newcommand{\inc}{\int_{\C}}
\newcommand{\inr}{\int_{\R}}
\newcommand{\dla}{d\lambda_\alpha}
\newcommand{\be}{\mathcal{B}}
\newcommand{\FF}{\mathcal F}

\newcommand{\re}{{\rm Re}}

\begin{document}

\title[Sarason's Problem]{Sarason's Toeplitz product problem\\ for a class of Fock spaces}

\author{ H\'el\`ene Bommier-Hato,  El Hassan Youssfi, and Kehe Zhu} 

\address{Bommier-Hato:  Faculty of Mathematics,
University of Vienna,
Oskar-Morgenstern-Platz 1, 1090 Vienna, Austria\   \& \ Aix-Marseille Universit\'e, I2M UMR CNRS 7373, 39~Rue F.~Joliot-Curie,
13453~Marseille Cedex~13, France}
\email{ helene.bommier@gmail.com}

\address{Youssfi: Aix-Marseille Universit\'e, I2M UMR CNRS 7373, 39~Rue F.~Joliot-Curie,
13453~Marseille Cedex~13, France}
\email{ el-hassan.youssfi@univ-amu.fr } 

\address{Kehe Zhu, Department of Mathematics, Shantou University, Guangdong Province, China 515063}
\email{kzhu@math.albany.edu}

\thanks{The first author is supported by the FWF project P 30251-N35.}



\maketitle

 
\begin{abstract} 
Sarason's Toeplitz product problem asks when the operator $T_uT_{\overline v}$ is bounded on various Hilbert spaces of 
analytic functions, where $u$ and $v$ are analytic. The problem is highly nontrivial for Toeplitz operators on the Hardy 
space and the Bergman space (even in the case of the unit disk). In this paper, we provide a complete solution 
to the problem for a class of Fock spaces on the complex plane. In particular, this generalizes an earlier result of Cho, 
Park, and Zhu.
\end{abstract}

\section{Introduction}

Let $\D$ be the open unit disk in the complex plane $\C$ and let $\T=\partial\D$ denote the unit circle. The Hardy space 
$H^2$ consists of functions $f\in L^2(\T)$ such that its Fourier coefficients satisfy $\hat f_n=0$ for all $n<0$. Given a
function $\varphi\in L^2(\T)$, the Toeplitz operator $T_\varphi:H^2\to H^2$ is densely defined by $T_\varphi f=P(\varphi f)$,
where $P:L^2(\T)\to H^2$ is the Riesz-Szeg\"o projection.

The original problem that Sarason proposed in \cite{Sarason} was this: characterize the pairs of outer functions $u$ and $v$ 
in $H^2$ such that the operator $T_uT_{\overline v}$ is bounded on $H^2$. Inner factors can easily be disposed of, so it was 
only necessary to consider outer functions in the Hardy space case. It was further observed in \cite{Sarason} that 
a necessary condition for the boundedness of $T_uT_{\overline v}$ on $H^2$ is that
$$\sup_{w\in\D}P_w(|u|^2)P_w(|v|^2)<\infty,$$
where $P_w(f)$ means the Poisson transform of $f$ at $w\in\D$. In fact, the arguments in \cite{Sarason} show that
\begin{equation}
\sup_{w\in\D}P_w(|u|^2)P_w(|v|^2)\le4\|T_uT_{\overline v}\|^2.
\label{eq1}
\end{equation}

Let $A^2$ denote the Bergman space consisting of analytic functions in $L^2(\D,dA)$, where $dA$ is ordinary area
measure on the unit disk. If $P:L^2(\D,dA)\to A^2$ is the Bergman projection, then Toeplitz operators $T_\varphi$ on
$A^2$ are defined by $T_\varphi f=P(\varphi f)$. Sarason also posed a similar problem in \cite{Sarason} for the Bergman 
space: characterize functions $u$ and $v$ in $A^2$ such that the Toeplitz product $T_uT_{\overline v}$ is bounded on 
$A^2$. It was shown in \cite{SZ} that
\begin{equation}
\sup_{w\in\D}\widetilde{|u|^2}(w)\widetilde{|v|^2}(w)\le 16\|T_uT_{\overline v}\|^2
\label{eq2}
\end{equation}
for all functions $u$ and $v$ in the Bergman space $A^2$, where $\widetilde f(w)$ is the so-called Berezin transform of
$f$ at $w$. This provides a necessary condition for the boundedness of $T_uT_{\overline v}$ on $A^2$ in terms of the
Berezin transform.

The Berezin transform is well defined in many other different contexts. In particular, the classical Poisson transform is
the Berezin transform in the context of the Hardy space $H^2$. So the estimates in (\ref{eq1}) and
(\ref{eq2}) are in exactly the same spirit. Sarason stated in \cite{Sarason} that ``it is tempting to conjecture that'' 
$T_uT_{\overline v}$ is bounded on $H^2$ or $A^2$ if and only if $\widetilde{|u|^2}(w)\widetilde{|v|^2}(w)$ is a 
bounded function on $\D$. It has by now become standard to call this ``Sarason's conjecture for Toeplitz products".

It turns out that Sarason's conjecture is false for both the Hardy space and the Bergman space of the unit disk, and
the conjecture fails in a big way. See \cite{APR,N} for counter-examples. In these cases, Sarason's problem is naturally 
connected to certain two-weight norm inequalities in harmonic analysis, and counter-examples for Sarason's conjecture
were constructed by means of the dyadic model approach in harmonic analysis.

Another setting where Toeplitz operators have been widely studied is the Fock space. More specifically, we
let $\FF^2$ be the space of all entire functions $f$ on $\C $ that are square-integrable with respect to the 
Gaussian measure 
$$d\lambda(z)=\frac{1}{\pi}\,e^{-|z|^2}\,dA(z).$$
The function 
$$K(z,w)=e^{z\overline{w}},\quad z,w\in\C,$$
is the reproducing kernel of $\FF^2$ and the orthogonal projection $P$ from $L^2(\C,d\lambda)$ 
onto $\FF^2$ is the integral operator defined by
$$Pf(z) =\int_{\C}K(z,w)f(w)d\lambda(w), \quad z\in\C.$$
If $\varphi$ is in $L^2(\C,d\lambda)$ such that the function $z\mapsto \varphi(z)K(z,w)$ belongs to
$L^1(\C,d\lambda)$ for any $w\in \C$, we can define the Toeplitz operator $T_\varphi$ with symbol $\varphi$ by 
$T_\varphi f=P(\varphi f)$, or
$$ T_\varphi f(z)= \int_{\C}K(z,w)\varphi(w)f(w)\,d\lambda(w),\quad z\in\C,$$
when
$$f(w)=\sum^{N}_{k=1}c_k K(w, c_k)$$
is a finite linear combination of kernel functions.  Since the set of all finite linear combinations of kernel functions is 
dense in $\FF^2$, the operator $T_\varphi$ is densely defined and $T_\varphi f$ is an entire function. 
See \cite{Z2} for basic information about the Fock space and Toeplitz operators on it.

In a recent paper \cite{CPZ}, Cho, Park and Zhu solved Sarason's problem for the Fock space. More specifically, they 
obtained the following simple characterization for $T_{u}T_{\overline{v}}$ to be bounded on $\FF^2$: if $u$ and 
$v$ are functions in $\FF^2$, not identically zero, then $T_{u}T_{\overline{v}}$  is bounded on $\FF^2$ 
if and only if $u=e^{q}$ and $v=ce^{-q}$, where $c$ is a nonzero constant and $q$ is a complex linear polynomial. As a 
consequence of this, it can be shown that Sarason's conjecture is actually true for Toeplitz products on $\FF^2$; 
see Section 5 below.

In this paper, we consider the weighted Fock space $\FF_m^2$, consisting of all entire functions in 
$L^2(\C,d\lambda_{m})$, where  $d\lambda_{m}$ are the generalized Gaussian measure defined by
$$d\lambda_{m}(z)=e^{-|z|^{2m}}\,dA(z),\quad m\geq 1.$$
Toeplitz operators on $\FF^2_{m}$ are defined exactly the same as the cases above, using the orthogonal
projection $P:L^2(\C,d\lambda_{m})\to\FF^2_{m}$.

We will solve Sarason's problem and prove Sarason's conjecture for the weighted Fock
spaces $\FF_{m}^2$. Our main result can be stated as follows.
 
\begin{mthm}
Let $u$ and $v$ be in $\FF^2_{m}$, not identically zero. The following conditions are equivalent:
\begin{enumerate}
\item The product $T=T_uT_{\overline{v}}$ is bounded on $\FF^2_{m}$.
\item  There exist a polynomial $g$ of degree at most $m$ and a nonzero complex constant $c$ such 
that $u(z)=e^{g(z)}$ and $v(z)=ce^{-g(z)}$.
\item The product $\widetilde{|u|^2}(z)\widetilde{|v|^2}(z)$ is a bounded function on $\C$.
\end{enumerate}
Furthermore, in the affirmative case, we have the following estimate of the norm:
$$\|T\| \leq C_1 e^{C_2\|g\|_{H^{2}}^{2}},$$
where $\|g\|_{H^{2}}$ is the norm in the Hardy space of the unit disc, and $C_1$ and $C_2$ are positive
constants independent of $g$.
\end{mthm}
Let us mention that \cite{Isr} contains partial results related to Sarason's conjecture on the Fock space.
The arguments in \cite{CPZ} depend on the explicit form of the reproducing kernel and the Weyl operators induced
by translations of the complex plane. Both of these are no longer available for the spaces $\FF^2_{m}$: there
is no simple formula for the reproducing kernel of $\FF^2_{m}$ and the translations on the complex
plane do not induce nice operators on $\FF^2_{m}$. Therefore, we need to develop new techniques to tackle
the problem.

\section{Preliminary estimates}

In this section we recall some properties of the Hilbert space $\FF_{m}^2$. It was shown in \cite{BY} that the 
reproducing kernel of $\FF^2_{m}$ is given by the formula
\begin{equation}
K_{m}(z,w)=\frac{m}{\pi}\sum^{+\infty}_{k=0}\frac{\left(z\overline{w}\right)^k}{\Gamma
\left(\frac{k+1}{m}\right)}.
\label{eq3}
\end{equation}
In terms of the Mittag-Leffler function
$$E_{\gamma,\beta}(z)=\sum_{k=0}^\infty\frac{z^k}{\Gamma(\gamma k+\beta)},\qquad \gamma,\beta>0, $$
we can also write
\begin{equation}  
K_m(z,w)=\frac{m}{\pi} \; E_{\frac1m,\frac1m}(z\overline{w}).   
\label{eq4}   
\end{equation} 

Recall that the asymptotics of the Mittag-Leffler function $E_{1/m,1/m}(z)$ as $|z|\to+\infty$ are given by
\begin{equation}
E_{\frac1m,\frac1m}(z) = \begin{cases} m z^{m-1} e^{z^m} \left(1+o(1)\right), \quad & |\arg z|\le\frac\pi{2m}, \\
O(\frac1z), & \frac\pi{2m}<|\arg z|\le\pi  \end{cases}   
\label{eq5}
\end{equation}
for $m>\frac12$, and by
$$E_{\frac1m,\frac1m}(z) = m \sum_{j=-N}^N z^{m-1} e^{2\pi ij(m-1)}
e^{z^m e^{2\pi ijm}} + O(\tfrac1z), \quad -\pi<\arg z\le\pi,$$
for $0<m\le\frac12$, where $N$ is the integer satisfying $N<\frac1{2m}\le N+1$ and the powers $z^{m-1}$ and $z^m$ 
are the principal branches. See, for example, Bateman and Erdelyi~\cite{BE}, vol.~III, 18.1, formulas (21)--(22).

The asymptotic estimates of the Mittag-Leffler function $E_{\frac{1}{m},\frac{1}{m}}$ provide the following estimates for 
the reproducing kernel $K_m(z,w)$, which is a consequence of the results in \cite{BY} and Lemma 3.1 in \cite{SY}.

\begin{lemma}
For arbitrary points $x,r \in(0,+\infty)$ and $\theta\in(-\pi,\pi)$ we have
$$|K_m(x,r e^{i\theta})|\lesssim\begin{cases}  (xr)^{m-1}e^{(xr)^m\cos(m\theta)} &|\theta|\le \frac{\pi}{2m}\\
O\left( \frac{1}{xr}\right), &\frac{\pi}{2m}\le|\theta|<\pi\end{cases}$$
as $xr\to+\infty$. Moreover, there is a constant $c>0$ such that for all $|\theta|\le c\theta_0(xr)$ we have
$$|K_m(x,re^{i\theta})|\gtrsim(xr)^{m-1}e^{(xr)^{m}}$$
as $xr\to+\infty$, where $\theta_0(r)=r^{-\frac{m}{2}}/m$.
\label{1}
\end{lemma}

On several occasions later on we will need to know the maximum order of a function in $\FF^2_{m}$. For
example, if we have a non-vanishing function $f$ in $\FF^2_{m}$ and if we know that the order of $f$ is
finite, then we can write $f=e^q$ with $q$ being a polynomial. The following estimate allows us to do this.

\begin{lemma}
If $f\in\FF_{m}^2$, there is a constant $C>0$ such that
$$\left|f(z)\right|\leq C \left|z\right|^{m-1}e^{\frac{1}{2}|z|^{2m}}, \quad z\in\C.$$
Consequently, the order of every function in $\FF^2_{m}$ is at most $2m$.
\label{2}
\end{lemma}

\begin{proof}
By the reproducing property and Cauchy-Schwartz inequality, we have
$$\left|f(z)\right|=\left|\int_{\C}f(w)K_{m}(z,w)d\lambda_{m}(w)\right|
\leq \left\|f\right\| K_{m}(z,z)^{1/2}$$
for all $f\in\FF^2_{m}$ and all $z\in\C$. The desired estimate then follows from Lemma~\ref{1}.
See \cite{B} for more details.
\end{proof}

Another consequence of the above lemma is that, for any function $u\in \FF_{m}^2$, the Toeplitz 
operators $T_u$ and $T_{\overline{u}} $ are both densely defined on $\FF_{m}^2$.

\section{Sarason's problem for $\FF^2_{m}$}

In this section we prove the equivalence of conditions (1) and (2) in the main theorem stated in the introduction,
which provides a simple and complete solution to Sarason's problem for Toeplitz products on the Fock space 
$\FF^2_{m}$. We break the proof into several lemmas.

\begin{lemma}
Suppose that $u$ and $v$ are functions in $\FF^2_{m}$, each not identically zero, and that the operator 
$T=T_uT_{\overline v}$ is bounded on $\FF^2_{m}$. Then there exists a polynomial $g$ of degree at 
most $m$ and a nonzero complex constant $c$ such that $u(z)=e^{g(z)}$ and $v(z)=ce^{-g(z)}$.
\label{3}
\end{lemma}

\begin{proof}
If $T=T_{u}T_{\overline{v}}$ is bounded on $\FF^2_{m}$, then the  Berezin transform $\widetilde{T}$ 
is bounded, where
$$\widetilde{T}(z)=\left\langle T_{u}T_{\overline{v}}k_z,k_z \right\rangle,\quad z\in\C.$$
By the reproducing property of the kernel functions, it is easy to see that
$$\widetilde{T}(z)=u(z)\overline{v(z)}.$$
Since each $k_z$ is a unit vector, it follows from the Cauchy-Schwarz inequality that
$$\left|u(z)v(z)\right|=|\widetilde T(z)|\leq \left\|T\right\|$$
for all $z\in\C$. This together with Liouville's theorem shows that there exist a constant $c$ such that $uv=c$. 
Since neither $u$ nor $v$ is identically zero, we have $c\not=0$. Consequently, both $u$ and $v$ are non-vanishing. 

Recall from Lemma~\ref{2} that the order of functions in $\FF^2_{m}$ is at most $2m$, so there is a
polynomial of degree $d$,
$$g(z)=\sum^{d}_{k=0}a_k z^k,\qquad d\leq[2m],$$
such that $u=e^{g}$ and $v=ce^{-g}$. It remains to show that $d\le m$.

Since $T$ is bounded  on $\FF^2_{m}$, the function
$$F(z, w)= \frac{\langle T\left(K_m(\cdot, w)\right), K_m(\cdot, z)\rangle}{\sqrt{K_m(z,z)}\sqrt{K_m(w,w)}}$$
must be bounded on $\C^2$. On general reproducing Hilbert spaces, we always have
\begin{eqnarray*}
\langle T_uT_{\overline v}K_w, K_z\rangle&=&\langle T_{\overline v}K_w,T_{\overline u}K_z\rangle
=\langle\overline v(w)K_w,\overline u(z)K_z\rangle\\
&=&u(z)\overline v(w)K(z,w).
\end{eqnarray*}
It follows that
$$F(z, w)=\overline ce^{g(z)-\overline{g(w)}}\frac{K_m(z, w)}{\sqrt{K_m(z,z)}\sqrt{K_m(w,w)}}.$$
From Lemma \ref{1} we deduce that 
\begin{equation}
\left |F(z, w)\right|\gtrsim e^{\text{Re}\left(g(z)-g(w)\right)}e^{-\frac{1}{2}\left(|z|^{m}-|w|^{m}\right)^2}
\label{eq6}
\end{equation}
for all $|\arg(z\bar w)|\le c\theta_0(|zw|)$ as $|zw|$ grows to infinity. 

Choose $x>0$ sufficiently large and set
$$z(x)=xe^{i \frac{\pi}{2d}}e^{-i\frac{\arg\left(a_d\right)}{d}},$$
and
$$w (x)=xe^{i\frac{\pi}{2d}}e^{-i\frac{\arg\left(a_d\right)+\frac{c}{2mx^m}}{d}}.$$
Since 
$$\theta_0(|z(x)w(x)|) = \frac{1}{mx^m},$$
we can apply \eqref{eq6} to $z(x)$ and $w(x)$ to get
\begin{equation}
e^{\text{Re}\left(g(z(x)) - g(w(x))\right)}\lesssim\sup_{(z, w)\in\C^2}\left|F(z, w )\right|<\infty
\label{eq7}
\end{equation}   
as $x$ grows to infinity.  On the other hand, a few computations show that
\begin{eqnarray*}
\text{Re} \left(g(z(x)) - g(w(x))\right)&=&\sum_{j=0}^{d}x^{j}\text{Re}\left(a_j e^{i j\frac{\pi}{2d}
-i\frac{j}{d}\arg\left[a_d\right)}\left(1-e^{-i \frac{cj}{2md x^m}} \right)\right] \\
&=&|a_d| x^d \sin\left( \frac{c}{2mx^m}\right)  + g_{d-1}(x),
\end{eqnarray*}
where 
$$ \aligned g_{d-1}(x) & = \sum_{j=0}^{d-1} x^{j} \text{Re} \left(a_j e^{i \frac{j\pi}{2d} - i\frac{j}{d}
\arg\left(a_d\right)}  \left( 1-  e^{-i  \frac{cj}{2md x^m}} \right) \right)\\
& = -\sum_{j=0}^{d-1} |a_j|x^{j}\sin\left(\frac{j\pi}{2d}+ \arg a_j -\frac{j}{d}\arg\left(a_d\right) \right)   
\sin\frac{cj}{2md x^m}\\
& \ \  +   \sum_{j=0}^{d-1} |a_j|x^{j}\cos\left[\frac{j\pi}{2d}+ \arg a_j  -\frac{j}{d}\arg\left(a_d\right)\right]  
\left[1- \cos\frac{cj}{2md x^m}\right]\\
& \lesssim x^{d-1-m}.
\endaligned$$
Therefore, there exist some $x_0>0$ and $\delta>0$ such that
$$\text{Re} \left(g(z(x)) - g(w(x))\right)\geq\frac{\delta |a_d| x^d}{x^m} $$
for all $x\geq x_0$. Since $a_d \not = 0$, it follows from \eqref{eq7} that $d\leq m$.
\end{proof}

On several occasions later on we will need to estimate the integral
$$I(a)=\int_0^\infty e^{-\frac12r^{2m}+ar^d}r^N\,dr,$$
where $m>0$, $0\le d\le m$, $N>-1$, and $a\ge0$.  

First, suppose $a>1.$ By various changes of variables, we have
\begin{eqnarray*}
I(a)&=&\int_0^1e^{-\frac12r^{2m}+ar^d}r^N\,dr+\int_1^\infty e^{-\frac12r^{2m}+ar^d}r^N\,dr\\
&\le&e^a\int_0^1r^N\,dr+\int_1^\infty e^{-\frac12r^{2m}+ar^m}r^N\,dr\\
&=&\frac{e^a}{N+1}+e^{\frac{a^2}2}\int_1^\infty e^{-\frac12(r^m-a)^2}r^N\,dr\\
&=&\frac{e^a}{N+1}+\frac{e^{\frac{a^2}2}}m\int_1^\infty e^{-\frac12(t-a)^2}t^{\frac{N+1}m-1}\,dt.
\end{eqnarray*}
If $\frac{N+1}m-1 \leq 0$, then 
$$I(a)\leq \frac{e^a}{N+1}+\frac{ \sqrt{2\pi}}{m} e^{\frac{a^2}2}
\leq \left(\frac{\sqrt{e}}{N+1}+\frac{ \sqrt{2\pi}}{m} \right) e^{\frac{a^2}2}.$$

Otherwise, we have $\frac{N+1}m-1 > 0. $ Using the fact that $u \mapsto u^{\frac{N+1}m-1}$ is increasing,   
we see that
$$\int_{- \frac a 2}^{ \frac a 2}   e^{-\frac{t^2}{ 2}}(t+a)^{\frac{N+1}m-1}\,dt\le\left(\frac{3a}{2}\right)^{\frac{N+1}m-1}  
\int_{-\frac a 2}^{ \frac a 2}   e^{-\frac{t^2}{ 2}}\,dt\le\sqrt{2\pi}\left(\frac{3a}{2}\right)^{\frac{N+1}m-1}.$$
For the same reason we also have
\begin{eqnarray*}
\int_{\frac a 2}^{+\infty} e^{-\frac{t^2}{ 2}}(t+a)^{\frac{N+1}m-1}\,dt  & \leq &
\int_{\frac a 2}^{+\infty} e^{-\frac{t^2}{ 2}}(3t)^{\frac{N+1}m-1}\,dt  \\
& \leq &3^{\frac{N+1}m-1}    \int_{0}^{+\infty} t^{\frac{N+1}m-1}e^{-\frac{t^2}{ 2}}\,dt  \\
&  = & \frac{\sqrt{2}}{2}\left(3\sqrt{2}\right)^{\frac{N+1}m-1}    \int_{0}^{+\infty} u^{\frac{N+1}{2m}-1}e^{-u}\,dt  \\
& = &\frac{\sqrt{2}}{2}\left(3\sqrt{2}\right)^{\frac{N+1}m-1} \Gamma\left(\frac{N+1}{2m}\right).
\end{eqnarray*}
In the case when $1-a <- \frac a 2$ (or equivalently $a>2$), 
\begin{eqnarray*}
\int_{1-a}^{- \frac a 2}  e^{-\frac{t^2}{ 2}}(t+a)^{\frac{N+1}m-1}\,dt & \leq &\left(\frac{a}{2}\right)^{\frac{N+1}m-1} 
\int_{1-a}^{- \frac a 2}  e^{-\frac{ t^2}{ 2}}\,dt \\
& \leq &\left(\frac{a}{2}\right)^{\frac{N+1}m-1} \int_{1-a}^{- \frac a 2}  e^{\frac{ a t}{ 4}}\,dt\\
& \leq &\left(\frac{a}{2}\right)^{\frac{N+1}m-1} \frac  4 a e^{\frac{ -a^2 }{ 8}}\\
& \leq &2\left(\frac a2\right)^{\frac{N+1}m-1}.
\end{eqnarray*}
It follows that there exists a constant $C=C(m,N)>0$ such that
$$\int_1^\infty e^{-\frac12(t-a)^2}t^{\frac{N+1}m-1}\,dt=\int_{1-a}^\infty e^{-\frac{t^2}{ 2}}(t+a)^{\frac{N+1}m-1}\,dt
\le C \left(1+a\right)^{\frac{N+1}m-1}$$
for $\frac{N+1}m-1>0$. It is then easy to find another positive constant $C=C(m,N)$, independent 
of $a$, such that
$$I(a)\le C\left(1+a\right)^{\frac{N+1}m -1} e^{\frac{a^2}{2}}$$
for all $a\ge1$ and $\frac{N+1}m-1>0$.  Therefore,
\begin{equation}
\int_0^\infty e^{-\frac12r^{2m}+ar^d}r^N\,dr\le C\left(1+a\right)^{\max\left(0,\frac{N+1}m -1\right)} 
e^{\frac{a^2}{2}}
\label{eq8}
\end{equation}
for all $a\ge1$. Since $I(a)$ is increasing in $a$, the estimate above holds for $0\le a\le1$ as well.

\begin{lemma}
For any $m>0, \delta >0, R\ge1, N>-1$,  and $ p \geq 0$, we can find a constant $C  >0$ (depending on 
$R, \delta, p, N, m$ but not on $a, d,x$) such that 
\begin{eqnarray*}
&&x^{N+ 1-p }\int_{\frac{R}{x^2}}^{+\infty}e^{-\frac{x^{2m}}{2}\left(1+r^{2m}\right)+ax^d(1+\delta r^{d}) }r^N\,dr\\
&&\qquad\le C\left(1+a\right)^{\max\left(0,\frac{N+p+1}m-1\right)}e^{\frac{1+\delta^2}{2}\,a^2}
\end{eqnarray*}
and
$$x^{m}\int_{\frac{R}{x^2}}^{+\infty}e^{-\frac{x^{2m}}{2}\left(1-r^{m}\right)^2+ax^d(1-r^d)}r^\frac{m}{2}\,dr
\le C(1+a)e^{\frac{a^2}{2}}$$
for all $x>0$, $a>0$, and $0\le d\le m$.
\label{4}
\end{lemma}

\begin{proof}   
Let $I=I(m,N, p, R,x,a,d)$ denote the first integral that we are trying to estimate. If $x\ge1$, we have
\begin{eqnarray*}
I&=& x^{N +1 - p}e^{-\frac{x^{2m}}2+ax^d}\int_{\frac R{x^2}}^\infty e^{-\frac{(xr)^{2m}}2+a\delta(xr)^d}r^N\,dr\\
&\le&x^{-p }e^{-\frac{x^{2m}}2+ax^m}\int_{\frac Rx}^\infty e^{-\frac{r^{2m}}2+a\delta r^d}r^N\,dr\\
&\le& e^{-\frac12(x^m-a)^2+\frac{a^2}2}\int_{\frac Rx}^\infty\frac{r^{p}}{R^{p}} e^{-\frac12r^{2m}+a \delta r^d}r^N\,dr\\
&\le& \frac{e^{\frac{a^2}2}}{R^{p}}Ê\int_{\frac Rx}^\infty e^{-\frac{r^{2m}}2+a \delta r^d}r^{N + p}\,dr.
\end{eqnarray*}
The desired result then follows from (\ref{eq8}).

If $0<x<1$, we have
\begin{eqnarray*}
I&=&x^{N+1-p}e^{-\frac{x^{2m}}2+ax^d}\int_{\frac R{x^2}}^\infty e^{-\frac{(xr)^{2m}}2+a \delta(xr)^d}r^N\,dr\\
&\le&e^ax^{-p}\int_{\frac Rx}^\infty e^{-\frac{r^{2m}}2+a  \delta r^d}r^N\,dr\\
&\le&\frac{e^{\frac{a^2}2+1}}{R^{p}}\int_{\frac Rx}^\infty e^{-\frac{r^{2m}}2+a  \delta r^d}r^{N+p}\,dr.
\end{eqnarray*}
The desired estimate follows from (\ref{eq8}) again.

To prove the second part of the lemma, denote by $J=J(m, d,R,x,a)$ the second integral that we are trying 
to estimate. Then it is clear from a change of variables that for $0<x<1$ we have 
\begin{eqnarray*}
J(m, d,R,x,a)&=&x^{\frac{m}{2}-1}\int_{\frac{R}{x}}^{+\infty}e^{-\frac{1}{2}\left(x^{m}-r^{m}\right)^2
+a(x^d-r^{d}) }r^\frac{m}{2}\,dr \\
&\le&\frac{e^a}{R}x^{\frac{m}{2}}\int_{\frac{R}{x}}^{+\infty} e^{-\frac{1}{2}\left(x^{2m}-2(xr)^m
+r^{2m}\right)}r^{\frac{m}{2}+1}\,dr\\
&\le&\frac{e^a}{R}\int_0^{+\infty}e^{-\frac{r^{2m}}{2}+r^m}r^{\frac{m}{2}+1}\,dr\\
&=& Ce^a\le C'(1+a)e^{\frac{a^2}{2}},
\end{eqnarray*}
where the constants $C$ and $C'$ only depend on $R$ and $m$.

Next assume that $x\geq 1$. In case $R \leq x^2$ we write $J= J_1+J_2$, where
$$J_1=J_1(m,d,R,x,a)=x^{m}\int_{\frac{R}{x^2}}^1e^{-\frac{x^{2m}}{2}\left(1-r^{m}\right)^2+ax^d(1-r^{d})}r^\frac{m}{2}\,dr,$$
and
$$J_2=J_2(m,d,R,x,a)=x^{m}\int_1^{+\infty}e^{-\frac{x^{2m}}{2}\left(1-r^{m}\right)^2+ax^d(1-r^{d})}r^\frac{m}{2}\,dr.$$ 
Otherwise we just use $J\leq J_2$.  So it suffices to estimate the two integrals above.
 
To handle $J_1(m, d, R,x,a)$, we fix $\varepsilon >0 $ and consider two cases. In the case $x^m\leq a(1+\varepsilon)$, 
we have
\begin{eqnarray*}
J_1(m, d,R,x,a)&\leq&x^{m}\int_{\frac{R}{x^2}}^1e^{-\frac{x^{2m}}{2}\left(1-r^{m}\right)^2+ax^m(1-r^m)}r^\frac m2\,dr\\
&\leq&a(1+\varepsilon)e^{\frac{a^2}2}\int_{\frac R{x^2}}^1e^{-\frac12\left(x^m(1-r^m)-a\right)^2}r^\frac{m}{2}\,dr\\
&\leq&a(1+\varepsilon)e^{\frac{a^2}{2}}.  
\end{eqnarray*}
When $x^m\geq a (1+ \varepsilon)$, we set $y=x^m$ and $\tau=(y-a)/2$.  Then we have
$$\tau \geq \frac{\varepsilon}{2(1+ \varepsilon)}\,y\to+\infty$$
as $y\to+\infty$. By successive changes of variables we see that
\begin{eqnarray*}
J_1(m, d,R,x,a)&\leq&x^{m}\int_{\frac{R}{x^2}}^1e^{-\frac{x^{2m}}2\left(1-r^m\right)^2+ax^m(1-r^m)}r^\frac m2\,dr\\
&=&\frac ym\int_0^{1-\frac{R^m}{y^2}}(1-r)^{\frac{1}{m}-\frac{1}{2}}e^{-\frac{y^2r^2}{2}+ayr}\,dr\\
&=&\frac1m\int_0^{y-\frac{R^m}{y}}\left(1-\frac{r}{y}\right)^{\frac{1}{m}-\frac{1}{2}}e^{-\frac{r^2}{2}+ar}\,dr\\
&=&\frac{e^{\frac{a^2}{2}}}{m}\int_{-a}^{y-a-\frac{R^m}y}\left(1-\frac ay-\frac ry\right)^{\frac1m-\frac12}e^{-\frac{r^2}2}\,dr. 
\end{eqnarray*}
This shows that for $1\leq m\leq 2$ we have
$$J_1\leq\frac{e^{\frac{a^2}{2}}}{m}\int_{-a}^{y-a-\frac{R^m}{y}}e^{-\frac{r^2}{2}}\,dr\leq\frac{\sqrt{2\pi}}{m}e^{\frac{a^2}{2}}.$$
Thus we suppose that $m>2$. Then 
\begin{eqnarray*}
\int_{-\tau}^{\tau }\left(1-\frac{a}{y}-\frac{r}{y}\right)^{\frac{1}{m}-\frac{1}{2}}e^{-\frac{r^2}{2}}\,dr 
&\le&\left(1-\frac{a}{y}-\frac{\tau}{y}\right)^{\frac{1}{m}-\frac{1}{2}}\int_{-\tau}^{\tau }e^{-\frac{r^2}{2}}\,dr\\
&=&\left(\frac{\tau}{2y}\right)^{\frac{1}{m}-\frac{1}{2}}\int_{-\tau}^{\tau }e^{-\frac{r^2}{2}}\,dr\\
&\le&\sqrt{2\pi}\left(\frac{\varepsilon}{4(1+\varepsilon)}\right)^{\frac{1}{m}-\frac{1}{2}}. 
\end{eqnarray*}
Moreover, in case $-a<-\tau$, we have 
\begin{eqnarray*}
\int_{-a}^{-\tau}\left(1-\frac{a}{y}-\frac{r}{y}\right)^{\frac{1}{m}-\frac{1}{2}}e^{-\frac{r^2}{2}}\,dv 
&\leq&\left(1-\frac{a}{y}+\frac{\tau}{y}\right)^{\frac{1}{m}-\frac{1}{2}}\int_{-a}^{-\tau}e^{-\frac{\tau |r|}{2}}\,dr\\
&\leq&2\left(\frac{3\varepsilon}{2(1+\varepsilon)}\right)^{\frac{1}{m}-\frac{1}{2}}\frac{e^{-\frac{\tau^2}{2}}}{\tau}\\
&\leq&4\left(\frac{3}{2}\right)^{\frac{1}{m}-\frac{1}{2}}\left(\frac{\varepsilon}{1+\varepsilon}\right)^{\frac{1}{m}-\frac{3}{2}}
e^{-\frac{\varepsilon^2}{8(1+\varepsilon)^2}}.
\end{eqnarray*}
Similarly, in case $y-a-\frac{R^m}{y}\geq\tau$, we have
\begin{eqnarray*}
\int_{\tau}^{y-a-\frac{R^m}{y}}\left[1-\frac{a}{y}-\frac{r}{y}\right]^{\frac{1}{m}-\frac{1}{2}}e^{ -\frac{r^2}{2}}\,dr 
&\leq&\left[\frac{R^m}{y^2}\right]^{\frac{1}{m}-\frac{1}{2}}\int_{\tau}^{y-a-\frac{R^m}{y}}e^{-\frac{\tau r}{2}}\,dr\\
&\leq&2R^{1- \frac{m}{2}}\left[\frac{\varepsilon}{2(1+\varepsilon)}\right]^{\frac{2}{m}-1}   
\tau^{-\frac{2}{m}}e^{-\frac{\tau^2}{2}}\\
\left(\text{since}\ \tau\geq\frac{\varepsilon}{2(1+\varepsilon)}\right)\ \ &\leq&4R^{1-\frac{m}{2}} 
\frac{1+\varepsilon}{\varepsilon}e^{ -\frac{ \varepsilon^2 }{8(1+\varepsilon)^2}}.
\end{eqnarray*}
The last three estimates yield
$$J_1\le C(1+a)e^{\frac{a^2}{2}}$$
for some $C >0$ that is independent of $x$ and $a$.
 
To establish the estimate for $J_2$, we perform a change of variables to obtain
$$J_2\leq x^m\int_{1}^{+\infty}e^{-\frac{x^{2m}}{2}\left(1-r^{m}\right)^2}r^\frac{m}{2}\,dr
=\frac{1}{m}\int_{0}^{+\infty}e^{-\frac{r^{2}}{2}}\left(\frac{r}{x^m}+1\right)^{\frac{1}{m}-\frac{1}{2}}\,dr. $$
If $m\geq 2$, we have
$$J_2 \leq \frac{1}{m}\int_{0}^{+\infty}e^{-\frac{r^{2}}{2}}\,dr,$$
and if $1\leq m<2$, we have
$$J_2\le\frac{1}{m}\int_{0}^{+\infty}e^{-\frac{r^{2}}{2}}(r+1)^{\frac{1}{m}-\frac{1}{2}}\,dr.$$
Therefore, $J_2\le C$ for some $C >0$ that is independent of $x$ and $a$. This completes the proof of the lemma.
\end{proof}

In the proof of the main theorem, we will have to estimate the following two integrals:
$$I(x,r)=\int_{|\theta|\leq\frac{\pi}{2m}}e^{-(xr)^m+2ar^{d}\sin^2\left(\frac{\theta d}{2}\right)}
|K_m(x,re^{i\theta})|\,d\theta,$$
and
$$J(x,r)=\int_{|\theta|\geq\frac{\pi}{2m}}e^{-(xr)^m+a(x^{d}+r^{d})}|K_m(x, re^{i\theta})|\,d\theta,$$
where $x,r,a\in(0,+\infty)$ and $0\le d\le m$.

\begin{lemma}  
For any $m>0$ there exist positive constants $C=C(m)$ and $R=R(m)$ such that
$$I(x,r)\le C(xr)^{m-1}\int_{0}^{1}e^{-\left((xr)^m-ar^{d}\right)t^2}\,dt$$
and
$$J(x,r)\le\frac{Ce^{-(xr)^m+a(x^d+r^d)}}{xr}$$
for all $a>0$, $0\le d\le m$, and $x>0$ with $xr>R$.
\label{5}
\end{lemma}

\begin{proof}
It follows from Lemma \ref{1} that there exist positive constants $C=C(m)$ and $R=R(m)$ such that
for all $a>0$ and $xr>R$ we have
\begin{eqnarray*}
I(x, r)&\le&C(xr)^{m-1}\int_{|\theta|\leq \frac{\pi}{2m}}e^{-(xr)^m + (xr)^m \cos(m\theta)+
2ar^{d}\sin^2\left(\frac{\theta d}{2}\right)}\,d\theta\\
&=&2C(xr)^{m-1}\int_{0}^{\frac{\pi}{2m}}e^{- 2(xr)^m\sin^2 
\left(\frac{m\theta}{2}\right)  + 2ar^{d}\sin^2\left(\frac{\theta d}{2}\right)}\,d\theta \\
&\leq&2C(xr)^{m-1}\int_{0}^{\frac{\pi}{2m}}e^{- 2(xr)^m\sin^2 
\left(\frac{m\theta}{2}\right)  + 2ar^{d}\sin^2\left(\frac{m\theta}{2}\right)}\,d\theta \\
&\leq&2C(xr)^{m-1}\int_{0}^{\frac{\pi}{2m}}e^{-2\left((xr)^m-ar^{d}\right)\sin^2\left(\frac{m\theta}{2}\right)}\,d\theta\\
&=&\frac{4C}{m}(xr)^{m-1}\int_{0}^{\frac{\sqrt{2}}{2}}e^{-2\left( (xr)^m-ar^{d}\right)t^2 }\frac{dt}{\sqrt{1-t^2}}\\
&\le&\frac{4\sqrt2C}{m}(xr)^{m-1}\int_{0}^{\frac{\sqrt{2}}{2}}e^{-2\left (xr)^m-ar^{d}\right)t^2}\,dt\\
&\le&\frac{4\sqrt2C}m(xr)^{m-1}\int_{0}^{1}e^{-\left((xr)^m-ar^{d}\right)t^2}\,dt.
\end{eqnarray*}
The estimate 
$$J(x, r)\le\frac{Ce^{-(xr)^m+a(x^{d}+r^{d})}}{xr},\qquad xr>R,$$
also follows from Lemma \ref{1}.
\end{proof}

\begin{lemma}
For any $m\ge1$ there exist constants $R=R(m)>1$ and $C=C(m)>0$ such that 
$$\int_{\frac{R}{x}}^{+\infty}e^{-\frac{1}{2}\left(x^{m}-r^{m}\right)^2+a(x^{d}-r^{d})}I(x, r)r\,dr
\le C\left(1+a\right)^{\frac{1}{m} -1}e^{a^2}$$
and
$$\int_{\frac{R}{x}}^{+\infty}e^{-\frac12\left(x^m-r^m\right)^2}J(x,r)r\,dr
\le C\left(1+a\right)^{\max\left(0,\frac2m-1\right)}e^{a^2}$$
for all $x>0$, $a>0$, and $0\le d\le m$.
\label{6}
\end{lemma}

\begin{proof}   
For convenience we write
$$A_I(x,r)=e^{-\frac12(x^m-r^m)^2+a(x^d-r^d)}I(x, r)r,$$
and
$$A_J(x,r)=e^{-\frac12(x^m-r^m)^2}J(x, r)r.$$
Let $R$ and $C$ be the constants from Lemma \ref{5}. In the integrands we have $r>R/x$, or $xr>R$, so
according to Lemma~\ref{5},
$$I(x,r)\le C(xr)^{m-1}\int_0^1e^{-(xr)^mt^2+ar^dt^2}\,dt.$$
If, in addition, $x\le1$, then
$$I(x,r)\le Cr^{m-1}e^{ar^d},$$
and
$$A_I(x,r)=e^{-\frac12(x^m-r^m)^2}e^{ax^d-ar^d}I(x,r)r\le Cr^me^ae^{-\frac12(x^m-r^m)^2}.$$
It follows that
\begin{eqnarray*}
\int_{\frac Rx}^\infty A_I(x,r)\,dr&\le&Ce^a\int_{\frac Rx}^\infty r^me^{-\frac12(x^m-r^m)^2}\,dr\\
&\le&Ce^a\int_0^\infty r^me^{-\frac12x^{2m}+x^mr^m-\frac12r^{2m}}\,dr\\
&\le&Ce^a\int_0^\infty r^me^{r^m-\frac12r^{2m}}\,dr\\
&\le&C'\left(1+a\right)^{\frac{1}{m}-1}e^{a^2}.
\end{eqnarray*}
for all $a>0$ and $0<x\le1$.

Similarly, if $x\le1$ (and $xr>R$), we deduce from Lemma \ref{5} and (\ref{eq8}) that
\begin{eqnarray*}
\int_{\frac Rx}^\infty A_J(x,r)\,dr&\le&\frac CR\int_{\frac Rx}^\infty e^{-\frac12(x^m-r^m)^2}e^{-(xr)^m+ax^d+ar^d}r\,dr\\
&\le&\frac{Ce^a}R\int_{\frac Rx}^\infty e^{-\frac12r^{2m}+ar^d}r\,dr\\
&\le&C'(1+a)^{\max\left(0,\frac{2}{m}-1\right)}e^{a^2}.
\end{eqnarray*}

Suppose now that $x\geq 1$ and $rx>R$. By Lemma \ref{5} again,
$$A_I(x, r)\le Cr(xr)^{m-1}e^{-\frac{1}{2}\left(x^{m}-r^{m}\right)^2+a(x^{d}-r^{d})}  
\int_{0}^{1}e^{-t^2\left( (xr)^m -ar^d\right)}\,dt.$$
Fix a sufficiently small $\varepsilon\in(0, 1)$. If $(xr)^m\ge ar^d(1+\varepsilon)$, then
\begin{eqnarray*}
\int_0^1e^{-t^2((xr)^m-ar^d)}\,dt&=&\frac1{\sqrt{(xr)^m-ar^d}}\int_0^{\sqrt{(xr)^m-ar^d}}e^{-s^2}\,ds\\
&\le&\frac1{\sqrt{(xr)^m-ar^d}}\int_0^\infty e^{-s^2}\,ds\\
&=&\frac{\sqrt\pi}2\frac{(xr)^{-\frac m2}}{\sqrt{1-(ar^d/(xr)^m)}}\\
&\le&\sqrt{\frac{\pi(1+\varepsilon)}{4\varepsilon}}\,(xr)^{-\frac m2},
\end{eqnarray*}
so there exists a constant $C=C(m)$ such that
$$A_I(x, r)\le Cr(xr)^{\frac{m}{2}-1}e^{-\frac{1}{2}\left(x^{m}-r^{m}\right)^2+a(x^{d}-r^d)}.$$
If $(xr)^m \leq ar^d(1+\varepsilon)$, we have 
\begin{eqnarray*}
A_I(x, r)&\lesssim&a^{\frac{m-1}{m}}r^{\frac{d(m-1)+m}{m}}e^{-\frac{1}{2}\left(x^{2m}+r^{2m}\right)+ax^{d}}
\int_{0}^{1}e^{(1-t^2)\left((xr)^m -ar^d\right)}\,dt\\
&\leq& a^{\frac{m-1}{m}} r^{\frac{d(m-1)+m}{m}}e^{-\frac{1}{2}\left(x^{2m}+r^{2m}\right)+a(x^d+\varepsilon r^d)}.
\end{eqnarray*}
It follows that 
\begin{eqnarray*}
\int_{\frac{R}{x}}^{+\infty}A_I(x, r)\,dr&\lesssim&x^{\frac{m}{2}-1}\int_{\frac{R}{x}}^{+\infty}
e^{-\frac{1}{2}\left(x^{m}-r^{m}\right)^2+a (x^{d}-r^d)}r^{\frac{m}{2}}\,dr\\
&&\ \ +a^{\frac{m-1}{m}}\int_{\frac{R}{x}}^{+\infty}r^{\frac{d(m-1)+m}{m}}e^{-\frac{1}{2}\left(x^{2m}+r^{2m}\right)
+a(x^{d}+\varepsilon r^d)}\,dr.
\end{eqnarray*}
The change of variables $r\mapsto xr$ along with the second part of Lemma \ref{4} shows that 
$$x^{\frac{m}{2}-1}\int_{\frac{R}{x}}^{+\infty}
e^{-\frac{1}{2}\left(x^{m} -r^{m}\right)^2+a(x^{d}-r^d)}r^{\frac{m}{2}}\,dr\le C(1+a)e^{\frac{a^2}{2}}. $$
Similarly, the change of variables $r\mapsto xr$ together with the first part Lemma \ref{4} shows that
$$\int_{\frac{R}{x}}^{+\infty}r^{\frac{d(m-1)+m}{m}}e^{-\frac12\left(x^{2m}+r^{2m}\right)+a(x^{d}+\varepsilon r^d)}\,dr
\le C(1+a)^{\frac{d(m-1)+1}m}e^{\frac{1+\varepsilon^2}2\,a^2}.$$
We may assume that $\varepsilon<1$. Then we can find a positive constant $C$ such that
$$a^{\frac{m-1}m}\int_{\frac{R}{x}}^{+\infty}r^{\frac{d(m-1)+m}{m}}e^{-\frac{1}{2}\left(x^{2m}+r^{2m}\right)+
a(x^{d}+\varepsilon r^d)}\,dr\le C(1+a)^{\frac1m-1}e^{a^2}.$$
It follows that 
$$\int_{\frac{R}{x}}^{+\infty}A_I(x, r)\,dr\le C\left(1+a\right)^{\frac{1}{m} -1}e^{a^2}$$
for some other positive constant $C$ that is independent of $a$ and $x$. This proves the first estimate of the lemma. 

To establish the second estimate of the lemma, we use Lemma \ref{5} to get
$$xA_J(x, xr)=x^2re^{-\frac{x^{2m}}{2}\left(1-r^{m}\right)^2}J(x, xr)
\le Ce^{-\frac{x^{2m}}{2}\left(1+r^{2m}\right)+ax^d(1+r^{d})}.$$
It follows from this and Lemma~\ref{4} that
$$\int_{\frac{R}{x}}^{+\infty}A_J(x, r)\,dr=x\int_{\frac{R}{x^2}}^{+\infty}A_J(x, xr)\,dr
\le C(1+a)^{\max\left(0,\frac{2}{m}-1\right)}e^{a^2}.$$
This completes the proof of the lemma.
\end{proof}

\begin{lemma}
If $u(z)=e^{g(z)}$ and $v(z)=e^{-g(z)}$, where $g$ is a polynomial of degree
at most $m$, then the operator $T=T_uT_{\overline v}$ is bounded on $\FF^2_{m}$.
\label{7}
\end{lemma}

\begin{proof}
To prove the boundedness of $T=T_uT_{\overline v}$, we shall use a standard technique known as 
Schur's test \cite[p.42]{Z}. Since
$$Tf(z)=\inc K_m(z,w)e^{g(z)-\overline{g(w)}}f(w)e^{-|w|^{2m}}\,dA(w),$$
we have
$$|Tf(z)|e^{-\frac12|z|^{2m}}\le\inc H_g(z,w)|f(w)|e^{-\frac12|w|^{2m}}\,dA(w),$$
where
$$H_g(z,w):=|K_m(z,w)|e^{-\frac{1}{2}(|z|^{2m}+|w|^{2m})+\text{Re}\left(g(z)-\overline{g(w)}\right)}.$$
Thus $T$ will be bounded on $\FF^2_{m}$ if the integral operator $S_g$ defined by
$$S_gf(z)=\inc \left(H_g(z,w) + H_g(w,z)\right)f(w)\,dA(w)$$
is bounded on $L^2(\C,dA)$. Let 
$$H_g(z)=\inc H_g(z,w)\,dA(w), \qquad z\in \C. $$ 
Since 
$$ H_{-g}(z)=\inc H_g(w,z)\,dA(w),$$
for all $z\in \C$, by Schur's test, the operator $S_g$ is bounded on $L^2(\C,dA)$ if we
can find a positive constant $C$ such that
$$H_g(z) + H_{-g}(z) \le C,\qquad z\in\C.$$

By the Cauchy-Schwarz inequality, we have
$$H_{g_1+g_2}(z)\leq\sqrt{H_{2g_1}(z)H_{2g_2}(z)}$$
for all $z\in \C$ and holomorphic polynomials $g_1$ and $ g_2$. Moreover, if 
$$U_\theta(z)=e^{i\theta} z, \qquad z\in \C, \theta \in [-\pi, \pi],$$ 
then
$$H_{g \circ U_\theta}=H_{g} \circ U_\theta$$
for all $z\in\C$, $\theta\in[-\pi, \pi]$, and holomorphic polynomials $g$. Therefore, we only need prove the 
theorem for $g(z)=az^d$ with some $a>0$ and $d\leq m$ and establish that 
\begin{equation}
\sup_{x\geq 0}H_g(x)\le C_1e^{C_2a^2},
\label{eq9}
\end{equation}
where $C_k$ are positive constants independent of $a$ and $d$ (but dependent on $m$). We will see
that $C_2$ can be chosen as any constant greater than $1$.

It is also easy to see  that we only need to prove (\ref{eq9}) for $x\ge1$. This will allow us to
use the inequality $x^d\le x^m$ for the rest of this proof.

For $R>0$ sufficiently large (we will specify the requirement on $R$ later) we write 
$$H_g(x) = \int_{|xw| \leq R} H_g(x,w)\,dA(w)+\int_{|xw| \geq R} H_g(x,w)\,dA(w).$$
We will show that both integrals are, up to a multiplicative constant, bounded above by $e^{(1+\varepsilon)a^2}$.

By properties of the Mittag-Leffler function, we have
$$|K_m(x, w)| \leq  \frac{m}{\pi}E_{\frac{1}{m}, \frac{1}{m}}(R):=C_R,\qquad |xw|\le R.$$
It follows that the integral
$$I_1=\int_{x|w|\le R}H_g(x,w)\,dA(w)$$
satisfies
\begin{eqnarray*}
I_1&=&\int_{x|w| \leq R} |K_m(z,w)|e^{-\frac{1}{2}(|z|^{2m}+|w|^{2m})+a\text{Re } \left(x^d -w^d\right) }\,dA(w)\\
&\le&C_R\int_{x|w|\leq R}e^{-\frac{1}{2}\left(x^{2m}+|w|^{2m}\right)+a\text{Re}\left(x^d-w^d\right) }\,dA(w)\\
&\le&C_R e^{-\frac{1}{2}x^{2m}+ ax^d }\int_{x|w| \leq R}e^{-\frac{|w|^{2m}}{2}+a |w|^d }\,dA(w)\\
&\le&2\pi C_R e^{-\frac{1}{2}x^{2m}+ ax^m}\int_{0}^{+\infty}e^{-\frac{r^{2m}}{2}+ar^d}r\,dr\\
&\le&2\pi C_R e^{\frac{a^{2}}{2} }\int_{0}^{+\infty} e^{-\frac{r^{2m}}{2}+ar^d}r\,dr\\
&\le&C(1+a)^{\max\left(0,\frac{2}{m} -1\right)}e^{a^2},
\end{eqnarray*}
where the last inequality follows from (\ref{eq8}).

We now focus on the integral
$$I_2=\int_{x|w| \geq R} H_g(x,w)\,dA(w).$$
Observe that for all $x$, $r$, and $\theta$ we have 
\begin{eqnarray*}
\text{Re}\left(x^{d}-r^{d}e^{id\theta}\right)&=&x^{d}-r^{d}\cos(d\theta) \\
&=&x^{d}-r^{d}+r^{d}\left(1-\cos(d\theta)\right)\\ 
&=&x^{d}-r^{d}+2r^{d}\sin^2\left(\frac{d\theta}{2}\right).
\end{eqnarray*}  
It follows from polar coordinates that
\begin{eqnarray*}
I_2&=&\int_{\frac{R}{x}}^{+\infty}\int_{-\pi}^{\pi}H_g(x,re^{i\theta})r\,d\theta\,dr\\
&=&\int_{\frac{R}{x}}^{+\infty}\int_{-\pi}^{\pi}e^{-\frac{1}{2}\left(x^{2m}+r^{2m}\right)+a(x^{d}-r^{d}
\cos(d\theta))} |K_m(x, r e^{i\theta})| r\,d\theta\,dr\\
&=&\int_{\frac Rx}^\infty e^{-\frac12(x^m-r^m)^2+a(x^d-r^d)-(xr)^m}\,r\,dr\int_{-\pi}^\pi
e^{2ar^d\sin^2(\frac{d\theta}2)}|K_m(x,re^{i\theta})|\,d\theta\\
&\le&\int_{\frac{R}{x}}^{+\infty}e^{-\frac{1}{2}\left(x^{m}-r^{m}\right)^2}\left(e^{a(x^{d}-r^{d})}I(x, r)
+J(x,r)\right)r\,dr, 
\end{eqnarray*}
where 
$$I(x,r)=\int_{|\theta|\leq\frac{\pi}{2m}}e^{-(xr)^m+2ar^{d}\sin^2\left(\frac{d\theta}{2}\right)} 
|K_m(x, r e^{i\theta})|\,d\theta,$$
and
$$J(x,r)=\int_{|\theta|\geq\frac{\pi}{2m}}e^{-(xr)^m+a(x^{d}+r^{d})}|K_m(x, r e^{i\theta})|\,d\theta.$$
By Lemma \ref{6}, there exists another constant $C>0$ such that 
$$I_2\le C(1+a)^{\max\left(0,\frac{2}{m} -1\right)}e^{a^2}.$$
Therefore,
$$\sup_{z\in\C} \int_{\C} H_g(z,w)\,dA(w)\le C(1+a)^{\max\left(0,\frac{2}{m} -1\right)}e^{a^2}$$
for yet another constant $C$ that is independent of $a$ and $d$. Similarly, we also have
$$\sup_{z\in\C} \int_{\C} H_{-g}(z,w)\,dA(w)\le C(1+a)^{\max\left(0,\frac{2}{m} -1\right)}e^{a^2}$$
This yields (\ref{eq9}) and proves the lemma.
\end{proof}

\section{Sarason's Conjecture for $\FF^2_{m}$}

In this section we show that Sarason's conjecture is true for Toeplitz products on the Fock type space 
$\FF^2_{m}$. More specifically, we will prove that condition (3) in the main theorem stated in the 
introduction is equivalent to conditions (1) and (2). Again we will break the proof down into several lemmas.

\begin{lemma}
Suppose $u$ and $v$ are functions in $\FF^2_{m}$, not identically zero, such that the operator
$T=T_uT_{\overline v}$ is bounded on $\FF^2_{m}$. Then the function $\widetilde{|u|^2}(z)\widetilde{|v|^2}(z)$ 
is bounded on the complex plane.
\label{8}
\end{lemma}

\begin{proof}
Since $T_uT_{\overline{v}}$ is bounded on $\FF^2_m$, the operator $\left(T_uT_{\overline{v}}\right)^*=
T_vT_{\overline{u}}$ and the products $\left(T_uT_{\overline{v}}\right)^*T_uT_{\overline{v}}$ and 
$\left(T_vT_{\overline{u}}\right)^*T_vT_{\overline{u}}$ are also bounded on $\FF^2_m$. 
Consequently, their Berezin transforms are all bounded functions on $\C$. 

For any $z\in \C$ we let $k_z$ denote the normalized reproducing kernel of $\FF^2_m$ at $z$. Then
\begin{eqnarray*}
\left\langle\left(T_uT_{\overline{v}}\right)^*T_uT_{\overline{v}} k_z,k_z\right\rangle&=&
\left\langle T_uT_{\overline{v}} k_z,T_uT_{\overline{v}} k_z \right\rangle\\
&=&\left\langle u\overline{v(z)} k_z, u\overline{v(z)} k_z\right\rangle\\
&=&\left|v(z)\right|^2 \widetilde{|u|^2}(z)
\end{eqnarray*}
is bounded on $\C$. Similarly $\left|u(z)\right|^2 \widetilde{|v|^2}(z)$ is bounded on $\C$. By the proof of 
Lemma~\ref{3}, the product $uv$  is a non-zero complex constant, say, $u(z)v(z)=C$. It follows that the function
$$\widetilde{|v|^2}(z) \widetilde{|u|^2}(z)=\left|u(z)\right|^2 \widetilde{|v|^2}(z)\left|v(z)\right|^2 
\widetilde{|u|^2}(z)\frac{1}{|C|^2}$$
is bounded as well.
\end{proof}

To complete the proof of Sarason's conjecture, we will need to find a lower bound for the function
$$\be(z)=\widetilde{|v|^2}(z) \left|u(z)\right|^2,$$
where  $u=e^g$, $v=e^{-g}$, and $g$ is a polynomial of degree $d$.  We write 
$$g(z)=a_d z^d+g_{d-1}(z),$$
where 
$$a_d= a e^{i\alpha_d},\qquad a>0,$$
and
$$g_{d-1}(z)=\sum^{d-1}_{l=0}a_l z^l.$$
In the remainder of this section we will have to handle several integrals of the form
$$I(x)=\int_J S_x(r)e^{-g_x(r)}dr,$$
where $S_x$ and $g_x$ are $\mathcal{C}^3$-functions on the interval $J$, and the real number $x$ tends 
to $+\infty$. We will make use of the following variant of the Laplace method (see \cite{SY}).

\begin{lemma}
Suppose that 
\begin{enumerate}
\item[(a)] $g_x$ attains its minimum at a point $r_x$, which tends to $+\infty$ as $x$ tends to $+\infty$, 
with $c_x=g''_x(r_x)>0$;
\item[(b)] there exists $\tau_x$ such that for $\left|r-r_x\right|<\tau_x$, $g''_x(r)=c_x(1+o(1))$ as $x$ 
tends to $+\infty$;
\item[(c)] for $\left|r-r_x\right|<\tau_x$, $S_x(r)\sim S_x(r_x)$;
\item[(d)] we have
$$\int_J S_x(r)e^{-g_x(r)}dr=(1+o(1))\int_{\left|r-r_x\right|<\tau_x} S_x(r)e^{-g_x(r)}\,dr.$$
\end{enumerate}
Then we have the following estimate
\begin{equation}
I(x)=\left(\sqrt{2\pi}+o(1)\right)\left[c_x\right]^{-1/2}S_x(r_x)e^{-g_x(r_x)},\qquad x\to+\infty.
\label{eq10}
\end{equation}
\label{9}
\end{lemma}

The computations in \cite{SY} ensure that, under the assumptions on $g_x$ and $S_x$, we have 
\begin{equation}
\int_{\left|r-r_x\right|>\tau_x}S_x(r)e^{-g_x(r)}\,dr\lesssim\left(c_x\tau_x\right)^{-1}
\int_{\left|t\right|>\tau_x}e^{-\frac{1}{3}\tau_xc_x t }\,dt.
\label{eq11}
\end{equation}
In particular, if one of the two conditions $c_x \tau_x ^2\to+\infty$ and $c_x \tau_x \to+\infty$
is satisfied, then hypothesis (d) in Lemma~\ref{9} holds.

The study of  $\be(z)$ will require some additional technical lemmas.

\begin{lemma}
For $z=x e^{i\phi}$, with $x>0$ and $e^{i(\alpha_d+d\phi)}=1$, we have
$$\be(z)\gtrsim\int^{+\infty}_{0}(rx)^{-\frac{m}{2}}r^{2m-1}e^{-h_x(r)}\,dr$$
as $x\to+\infty$, where
\begin{equation}
h_x(r)=\left(r^m-x^m\right)^2-2a\left(x^d-r^d\right)+C\left(r ^{d-1}+x^{d-1}+1\right),
\label{eq12}
\end{equation}
for some positive constant $C$.
\label{10}
\end{lemma}

\begin{proof}
It is easy to see that
$$\be(z)=\int_{\C}\left|K_m(w,z)\right|^2e^{2\re\left(g(z)-g(w)\right)}\left[K_m(z,z)\right]^{-1}e^{-|w|^{2m}}\,dA(w),$$
which, in terms of polar coordinates, can be rewritten as
$$\int^{+\infty}_{0}\int^{-\pi}_{\pi}\left|K_m(re^{i\theta},z)\right|^2e^{2\re\left(g(z)-g(re^{i\theta})\right)}
\left[K_m(x,x)\right]^{-1}e^{-r^{2m}}r\,dr\,d\theta.$$
By Lemma~\ref{1}, $\be(z)$ is greater than or equal to
$$\int^{+\infty}_{0}\int_{\left|\theta-\phi\right|\leq c\theta_0(rx)}\left|K_m( re^{i\theta},z)\right|^2
e^{2\re\left(g(z)-g(re^{i\theta})\right)}\left[K_m(x,x)\right]^{-1}e^{-r^{2m}}r\,dr\,d\theta.$$
This together with Lemma \ref{1} shows that
$$\be(z)\gtrsim\int^{+\infty}_{0}r^{2(m-1)}e^{-\left(r^m-x^m\right)^2}I(r,z)r\,dr,$$
where
$$I(r,z)=\int_{\left|\theta-\phi\right|\leq c\theta_0(rx)}e^{2\re\left(g(z)-g(re^{i\theta})\right)}\,d\theta.$$
Note that
\begin{eqnarray*}
I(r,z)&=&\int_{\left|\theta-\phi\right|\leq c\theta_0(rx)}e^{2\re\left[ae^{i\alpha_d}\left(x^de^{id\phi}-r^d 
e^{id\theta}\right)\right]+2\re\left[g_{d-1}(z)-g_{d-1}(re^{i\theta})\right]}\,d\theta\\
&=&\int_{\left|\theta-\phi\right|\leq c\theta_0(rx)}e^{2\re\left[ae^{i(\alpha_d+d\phi)}\left(x^d-r^d 
e^{id(\theta-\phi)}\right)\right]+2\re\left[g_{d-1}(z)-g_{d-1}(re^{i\theta})\right]}\,d\theta.
\end{eqnarray*} 	
The condition on  $\phi$ yields
$$I(r,z)=\int_{\left|\theta\right|\leq c\theta_0(rx)}e^{2\re\left[a\left(x^d-r^de^{id\theta}\right)\right]+
2\re\left[g_{d-1}(z)-g_{d-1}(re^{i(\theta+\phi)})\right]}\,d\theta.$$
Since
$$g_{d-1}(z)-g_{d-1}(re^{i(\theta+\phi)})=\sum^{d-1}_{l=0}a_l \left(x^le^{il\phi}-r^le^{il(\theta+\phi)}\right),$$
we have
$$\re\left[g_{d-1}(z)-g_{d-1}(re^{i(\theta+\phi)})\right]\geq -C\left(r^{d-1}+x^{d-1}+1\right)$$
for some constant $C$. It follows that
$$I(r,z)\geq e^{-C\left(r^{d-1}+x^{d-1}+1\right)}\int_{\left|\theta\right|\leq c\theta_0(rx)}
e^{2a\re\left[\left(x^d-r^de^{id\theta}\right)\right]}\,d\theta.$$

For  the integral we have
\begin{eqnarray*}
J(r,z)&:=&\int_{\left|\theta\right|\leq c\theta_0(rx)}e^{2a\re\left[\left(x^d-r^de^{id\theta}\right)\right]}\,d\theta\\
&=&\int_{\left|\theta\right|\leq c\theta_0(rx)}e^{2a\left(x^d-r^d\cos(d\theta)\right)}\,d\theta\\
&=&\int_{\left|\theta\right|\leq c\theta_0(rx)}e^{2a\left(x^d-r^d+(-\cos(d\theta)+1)r^d\right)}\,d\theta\\
&=&\int_{\left|\theta\right|\leq c\theta_0(rx)}e^{2a\left(x^d-r^d+2\left(\sin\left(\frac{d\theta}{2}\right)^2\right)r^d\right)}\,d\theta\\
&\geq&e^{2a\left(x^d-r^d\right)}\int_{\left|\theta\right|\leq c\theta_0(rx)}e^{4|a_d|\sin\left(\frac{ d\theta}{2}\right)^2r^d}\,d\theta\\
& \geq  &e^{2a\left(x^d-r^d\right)}\int_{\left|\theta\right|\leq c\theta_0(rx)}\,d\theta\\
&  \gtrsim&e^{2a\left(x^d-r^d\right)}(rx)^{-\frac{m}{2}},
\end{eqnarray*}
which completes the proof of the lemma.
\end{proof}

\begin{lemma}
Assume $d=2m.$ For $z=x e^{i\phi}$, where $x>0$  and $e^{i(\alpha_d+d\phi)}=1$, we have
$$\be(z)\gtrsim e^{(1+o(1))\frac{2a}{(1+2a)} x^{2m}},\qquad x\to+\infty.$$
\label{11}
\end{lemma}

\begin{proof}
For $x$ large enough, the function $h_x$ defined in (\ref{eq12}) is convex on some interval $[M_x,+\infty)$ and 
attains its minimum at some point $r_x$. In order to bound $\be(z)$ from below, we shall use the modified 
Laplace method from Lemma~\ref{9}. Since
\begin{equation}
h'_x(r)= 2mr^{m-1}\left(r^m-x^m\right)+2adr^{d-1}+C \left(d-1\right)r^{d-2},
\label{eq13}
\end{equation}
we have 
$$h'_x(r)=2m(1+2a)r^{2m-1}-2mx^mr^{m-1}+C(d-1)r^{d-2},$$
and
$$h''_x(r)=2m(2m-1)(1+2a)r^{2m-2}-2m(m-1)x^mr^{m-2}+C(d-1)(d-2)r^{d-3}.$$
Writing $h'_x(r_x)=0 $ and letting $x$ tend to $+\infty$, we obtain
$$m(1+2a)(r_x)^{2m-1}\sim mx^mr_x^{m-1},$$
or
\begin{equation}
r_x \sim(1+2a)^{-\frac{1}{m}}x.
\label{eq14}
\end{equation}

Thus there exists $\rho_x$, which tends to $0$ as $x$ tends to $+\infty$, such that
\begin{equation}
r_x =(1+2a)^{-\frac{1}{m}}x(1+\rho_x).
\label{eq15}
\end{equation}
When $x$ tends to $+\infty$, we have 
\begin{eqnarray*}
h_x(r_x)&\sim& (r_x^m-x^m)^2+2a(r_x^{2m}-x^{2m})\\
&\sim& (r_x^m-x^m)\left[ (r_x^m-x^m)+2a(r_x^m+x^m)\right]\\
&\sim& x^{2m}\left[(1+2a)^{-1}(1+\rho_x)^m  -1\right]\\
&&\ \left[(1+2a)^{-1}(1+\rho_x)^m-1  +2a\left((1+2a)^{-1}(1+\rho_x)^m+1\right)\right]\\
&\sim& -x^{2m} \frac{2a}{(1+2a)},
\end{eqnarray*}
or
\begin{equation}
-h_x(r_x)\sim x^{2m} \frac{2a}{(1+2a)}.
\label{eq16}
\end{equation}

In order to estimate $c_x:=h''_x(r_x)$, we compute that
\begin{eqnarray*}
h''_x(r_x)&\sim&2 m^2 (1+2a)^{-1+\frac{2}{m}}x^{2m-2}.
\end{eqnarray*}
 Thus we get
\begin{equation}
c_x\approx x^{2m-2}.
\label{eq17}
\end{equation}

For $r$ in a neighborhood of $r_x$ we set $r=(1+\sigma_x)r_x$, where $\sigma_x=\sigma_x(r)\to0$ 
as $x\to+\infty$; a little computation shows that 
$$ h''_x(r)\sim h''_x(r_x) $$
as $x\to+\infty$.
Taking $\tau_x=r_x^{1/2}$ and $|r-r_x|<\tau_x$, we have $h''_x(r)= (1+o(1))c_x$, so
$$h_x(r)-h_x(r_x)=\frac{1}{2}c_x(r-r_x)^2(1+o(1)).$$
Thus 
\begin{eqnarray*}
\int_{ |r-r_x|<\tau_x}e^{-\frac{1}{2}c_x(r-r_x)^2(1+o(1))}dr&=&\int_{ |t|<\tau_x}e^{-\frac{1}{2}c_xt^2(1+o(1))}\,dt\\
& \sim  &\frac{1}{\sqrt{c_x}}\int_{\left|y\right|<\tau_x\sqrt{c_x}}e^{-\frac{1}{2}y^2}\,dy\\
&\approx& \frac{1}{\sqrt{c_x}},
\end{eqnarray*}
because $c_x \tau_x^2\approx r_x^{2m-1}$ tends to $+\infty$ as $x$ tends to $+\infty$. Finally, the estimates
\begin{eqnarray*}
\be(z)&\gtrsim&\int_{ |r-r_x|<\tau_x}(rx)^{-\frac{m}{2}}r^{2m-1}e^{-h_x(r)}\,dr\\
&=&\int_{ |r-r_x|<\tau_x}(rx)^{-\frac{m}{2}}r^{2m-1}e^{-h_x(r_x)} e^{-\left[h_x(r)-h_x(r_x)\right]}\,dr\\
&=&e^{-h_x(r_x)}\int_{ |r-r_x|<\tau_x}(rx)^{-\frac{m}{2}}r^{2m-1}e^{-\frac{1}{2}c_x(r-r_x)^2(1+o(1))}\,dr\\
&\sim& e^{-h_x(r_x)}r_x^{\frac{3}{2}m-1}x^{-\frac{m}{2}}\int_{ |r-r_x|<\tau_x}e^{-\frac{1}{2}c_x(r-r_x)^2(1+o(1))}\,dr\\
&\approx& e^{-h_x(r_x)}r_x^{\frac{3}{2}m-1}x^{-\frac{m}{2}} \frac{1}{\sqrt{c_x}}
\end{eqnarray*}
along with (\ref{eq14}), (\ref{eq16}), and (\ref{eq17}) give the lemma.
\end{proof}

\begin{lemma}
Assume $d<2m.$ For $z=x e^{i\phi}$, with $x>0$  and $e^{i(\alpha_d+d\phi)}=1$, we have 
$$\be(z)\gtrsim e^{(1+o(1))\frac{a^2d^2}{m^2}x^{2d-2m}-C x^{d-1-m}},\qquad x\to+\infty$$
\label{12}
for some positive constant $C$
\end{lemma}

\begin{proof}
Let $\tau_x=o(x)$ be a positive real number that will be specified later. As in the proof of Lemma 10 we have
\begin{eqnarray*}
\be(z)&\gtrsim&\int^{+\infty}_{0}r^{2(m-1)}e^{-\left(r^m-x^m\right)^2}I(r,z)r\,dr\\
&\gtrsim&\int_{|r-x|\leq \tau_x}r^{2(m-1)}e^{-\left(r^m-x^m\right)^2}I(r,z)r\,dr,\\
\end{eqnarray*}
where 
$$I(r,z)=\int_{\left|\theta-\phi\right|\leq c\theta_0(rx)}e^{2\re\left(g(z)-g(re^{i\theta})\right)}\,d\theta.$$
There exists $c'>0$ such that for $|r-x|\leq \tau_x$ we have

\begin{eqnarray*}
I(r,z)&\geq& \int_{\left|\theta-\phi\right|\leq c'\theta_0(x^2)}e^{2\re\left(g(z)-g(re^{i\theta})\right)}\,d\theta\\
&=&\int_{\left|\theta\right|\leq c'\theta_0(x^2)}e^{2a  \re\left(x^d-r^d e^{id\theta}\right)+2\re\left[g_{d-1}(z)-g_{d-1}(re^{i\theta})\right]}\,d\theta\\
&=&\int_{\left|\theta\right|\leq c'\theta_0(x^2)}e^{2a  \re\left(x^d-r^d e^{id\theta}\right)-2\sum^{d-1}_{l=0} \left|a_l\right| \left|x^l-r^l e^{il\theta}\right|}\,d\theta.\\
\end{eqnarray*}
Now for  $|r-x|\leq \tau_x$, we write $r=(1+\sigma)x$, where $\sigma$ tends to $0$ as $x\to + \infty$. Thus for $0\leq l\leq d-1$ and  $\left|\theta\right|\leq c'\theta_0(x^2)$,  we obtain 
\begin{eqnarray*}
\left|x^l-r^l e^{il\theta}\right|^2&=&x^{2l}\left[1-2(1+\sigma)^l\cos(l\theta)+(1+\sigma)^{2l}\right]\\
&=&x^{2l}\left[ 1-2\left(1+l \sigma+O(\sigma^2)\right)\cos(l\theta)+1+2 l \sigma+O(\sigma^2)\right]\\
&=&x^{2l}\left[  2\left(1-\cos(l\theta)\right)\left(1+l \sigma\right)  +O(\sigma^2)\right]\\
&\lesssim& x^{2l}\left[ \sin^2\left(\frac{l\theta}{2}\right)+\sigma^2\right]\\
&\lesssim& x^{2l}\left[ \theta^2 +\sigma^2\right].\\
\end{eqnarray*}
Next choosing $|\sigma|\leq x^{-m}$, we get 
\begin{eqnarray*}
\left|x^l-r^l e^{il\theta}\right|&\lesssim& x^{2l}x^{-2m}\\
&\lesssim& x^{2(d-1)-2m}
\end{eqnarray*}
or
\begin{eqnarray*}
\left|x^l-r^l e^{il\theta}\right|&\lesssim& x^{d-1-m}.
\end{eqnarray*}
Thus there exists a positive constant $C$ such that for $|r-x|\leq \tau_x$ and  $\left|\theta\right|\leq c'\theta_0(x^2)$,
$$2\sum^{d-1}_{l=0} \left|a_l\right| \left|x^l-r^l e^{il\theta}\right|\leq C  x^{d-1-m}. $$
It follows that
\begin{eqnarray*}
I(r,z)&\geq&\int_{\left|\theta\right|\leq c'\theta_0(x^2)}e^{2a  \re\left(x^d-r^d e^{id\theta}\right) -C  x^{d-1-m} }\,d\theta\\
&\gtrsim&x^{-m}e^{2a  \re\left(x^d-r^d e^{id\theta}\right) -C  x^{d-1-m} }.
\end{eqnarray*}
Then
\begin{eqnarray*}
\be(z)&\gtrsim& \int_{|r-x|\leq \tau_x}r^{2m-1}e^{-\left(r^m-x^m\right)^2} x^{-m}e^{2a  \left(x^d-r^d \right) -C  x^{d-1-m} } \,dr\\
&=&   x^{-m} e^{-C  x^{d-1-m} }\int_{|r-x|\leq \tau}r^{2m-1}e^{-h_x(r)} \,dr,\\
\end{eqnarray*}
where 
$$h_x(r)=\left(r^m-x^m\right)^2-2a  \left(x^d-r^d \right). $$

It is easy to see that  $h_x$ attains its minimum at $r_x$ with $r_x\sim x   $ as $x\to+\infty$. Again we write 
\begin{equation}
r_x =x(1+\rho_x),
\label{eq18}
\end{equation}
where $\rho_x$ tends to $0$ as $x\to+\infty$. Using  the fact that $h'_x( r_x)=0$, we have
$$2mx^{2m-1}(1+\rho_x)^{m-1}\left[(1+\rho_x)^m-1\right]\sim -2a dx^{d-1}(1+\rho_x)^{d-1},$$
and
$$2mx^{2m-1} m\rho_x\sim -2ad x^{d-1}.$$
Therefore,
\begin{equation}
\rho_x\sim-\frac{ad}{m^2}x^{d-2m}.
\label{eq19}
\end{equation}
Since
\begin{eqnarray*}
h_x''(r)&=&2m(2m-1) r^{2m-2}-2m (m-1)x^m r^{m-2}\\
&&\quad+2a d(d-1)r^{d-2}
\end{eqnarray*}
and  $d< 2m$, we get
\begin{eqnarray*}
h_x''(r_x)&\sim&2m x^{2m-2}\left[(2m-1)(1+\rho_x)^{2m-2}-(m-1)(1+\rho_x)^{m-2 }\right]\\
&\sim& 2m^2 x^{2m-2}.
\end{eqnarray*}
Also,
\begin{eqnarray*}
h_x(r_x)&\sim& x^{2m}\left[ (1+\rho_x)^m-1\right]^2+2a x^d\left[(1+\rho_x)^d-1\right]+C(x^{d-1}+r_x^{d-1}+1)\\
&\sim& m^2 \rho_x ^2 x^{2m} +2a x^d d\rho_x \\
\end{eqnarray*}
It follows that
\begin{equation}
c_x\sim  2m^2 x^{2m-2},
\label{eq20}
\end{equation}
and 
\begin{equation}
-h_x(r_x)\sim\frac{a^2d^2}{m^2}x^{2d-2m}.
\label{eq21}
\end{equation}
Reasoning as in the proof of Lemma  \ref{11}, we arrive at 
$$\be(z)\gtrsim
\gtrsim  x^{-m}  e^{-C  x^{d-1-m} }   e^{-h_x(r_x)}x^{2m-1} \frac{1}{\sqrt{c_x}}.$$
The desired estimate then follows from (\ref{eq21}),  and (\ref{eq20}).
\end{proof}

\begin{lemma}
Suppose $u$ and $v$ are functions in $\FF^2_{m}$, not identically zero, such that 
$\widetilde{|u|^2}(z)\widetilde{|v|^2}(z)$ is bounded on the complex plane. Then there exists a nonzero
constant $C$ and a polynomial $g$ of degree at most $m$ such that $u(z)=e^{g(z)}$ and $v(z)=Ce^{-g(z)}$.
\label{13}
\end{lemma}

\begin{proof}
It is easy to check that for $u\in\FF^2_{m}$ we have
$$u(z)=\int_{\C}u(x)|k_z(x)|^2d\lambda_m(x)=\widetilde{u}(z).$$
Also, it follows from the Cauchy-Schwarz inequality that $\left|u(z)\right|^2\leq\widetilde{|u|^2}(z)$.
So if $\widetilde{|u|^2}(z) \widetilde{|v|^2}(z)$ is bounded on $\C$,  then $\be(z)$ and $|u(z)v(z)|^2$
are also  bounded. Consequently, $uv$ is a constant, and as in section 3, there is a non-zero constant $C$ and 
a polynomial $g$ such that $u=e^g$ and $v=Ce^{-g}$. The condition $u\in\FF^2_{m}$ implies that the 
degree $d$ of $g$ is at most $ 2m$; see Lemma~\ref{2}.

Without loss of generality we shall consider the case where $u(z)=e^{g(z)}$ and $v(z)=e^{-g(z)}$.  We will show that that the boundedness of $\be(z)$ implies $d\leq m$.
If $2m $ is an integer, Lemma~\ref{11} shows that we must have $d<2m$. Thus, in any case ($2m$ being an 
integer or not), a necessary condition is $d< 2m$. The desired result now follows from Lemma~\ref{12}.
\end{proof}

\section{Further Remarks}

In this final section we specialize to the case $m=1$ and make several additional remarks. For convenience
we will alter notation somewhat here. Thus for any $\alpha>0$ we let $F^2_\alpha$ denote the Fock space of
entire functions $f$ on the complex plane $\C$ such that
$$\inc|f(z)|^2\,d\lambda_\alpha(z)<\infty,$$
where
$$d\lambda_\alpha(z)=\frac\alpha\pi e^{-\alpha|z|^2}\,dA(z).$$
Toeplitz operators on $F^2_\alpha$ are defined exactly the same as before using the orthogonal projection
$P_\alpha:L^2(\C,d\lambda_\alpha)\to F^2_\alpha$.

Suppose $u$ and $v$ are functions in $F^2_\alpha$, not identically zero. It was proved in \cite{CPZ} that 
$T_uT_{\overline v}$ is bounded on the Fock space $F^2_\alpha$ if and only if there is a point $a\in\C$ such that
\begin{equation}
u(z)=be^{\alpha\overline az},\qquad v(z)=ce^{-\alpha\overline az},
\label{eq22}
\end{equation}
where $b$ and $c$ are nonzero constants. This certainly solves Sarason's problem for Toeplitz products on
the space $F^2_\alpha$. But the paper \cite{CPZ} somehow did not address Sarason's conjecture, which now of 
course follows from our main result.

We want to make two points here. First, the proof of Sarason's conjecture for $F^2_\alpha$ is relatively simple after
Sarason's problem is solved. Second, Sarason's conjecture holds for the Fock space $F^2_\alpha$ for completely 
different reasons than was originally thought, namely, the motivation for Sarason's conjecture provided in \cite{Sarason} 
for the cases of Hardy and Bergman spaces is no longer valid for the Fock space. It is therefore somewhat amusing 
that Sarason's conjecture turns out to be true for the Fock space but fails for the Hardy and Bergman spaces.

Suppose $u$ and $v$ are given by (\ref{eq22}). We have
\begin{eqnarray*}
\widetilde{|u|^2}(z)&=&\|fk_z\|^2=\inc\left|f(w)e^{\alpha w\bar z-(\alpha/2)|z|^2}\right|^2\,\dla(w)\\
&=&|b|^2e^{-\alpha|z|^2}\inc\left|e^{\alpha w(\bar a+\bar z)}\right|^2\,\dla(w)\\
&=&|b|^2e^{-\alpha|z|^2+\alpha|a+z|^2}\\
&=&|b|^2e^{\alpha(|a|^2+\overline az+a\overline z)}.
\end{eqnarray*}
Similarly,
$$\widetilde{|v|^2}(z)=|c|^2e^{\alpha(|a|^2-\overline az-a\overline z)}.$$
It follows that
$$\widetilde{|u|^2}(z)\widetilde{|v|^2}(z)=|bc|^2e^{2\alpha|a|^2}$$
is a constant and hence a bounded function on $\C$.

On the other hand, it follows from H\"older's inequality that we always have
$$|u(z)|^2\le\widetilde{|u|^2}(z),\qquad u\in F^2_\alpha, z\in\C.$$
Therefore, if $\widetilde{|u|^2}\widetilde{|v|^2}$ is a bounded function on $\C$, then
there exists a positive constant $M$ such that
$$|u(z)v(z)|^2\le\widetilde{|u|^2}(z)\widetilde{|v|^2}(z)\le M$$
for all $z\in\C$. Thus, as a bounded entire function, $uv$ must be constant, say
$u(z)v(z)=C$ for all $z\in\C$. Since $u$ and $v$ are not identically zero, we must have $C\not=0$.
Since functions in $F^2_\alpha$ must have order less than or equal to $2$, we can write $u(z)=e^{p(z)}$, where 
$$p(z)=az^2+bz+c$$
is a polynomial of degree less than or equal to $2$. But $u(z)v(z)$ is constant, so $v(z)=e^{q(z)}$, where
$$q(z)=-az^2-bz+d$$
is another polynomial of degree less than or equal to $2$.

We will show that $a=0$. To do this, we will estimate the Berezin transform $\widetilde{|u|^2}$ when $u$ is
a quadratic exponential function as given above. More specifically, for $C_1=|e^c|^2$, we have
\begin{eqnarray*}
\widetilde{|u|^2}(z)&=&C_1\inc\left|e^{a(z+w)^2+b(z+w)}\right|^2\,\dla(w)\\
&=&C_1\left|e^{az^2+bz}\right|^2\inc\left|e^{aw^2+(b+2az)w}\right|^2\,\dla(w).
\end{eqnarray*}
Write $b+2az=\alpha\bar\zeta$. Then it follows from the inequality $\widetilde{|F|^2}\ge|F|^2$ for 
$F\in F^2_\alpha$ again that
\begin{eqnarray*}
\widetilde{|u|^2}(z)&=&C_1\left|e^{az^2+bz}\right|^2e^{\alpha|\zeta|^2}\inc\left|e^{aw^2}k_\zeta(w)\right|^2\,\dla(w)\\
&\ge&C_1\left|e^{az^2+bz}\right|^2e^{\alpha|\zeta|^2}\left|e^{a\zeta^2}\right|^2.
\end{eqnarray*}
If we do the same estimate for the function $v$, the result is
$$\widetilde{|v|^2}(z)\ge C_2\left|e^{-az^2-bz}\right|^2e^{\alpha|\zeta|^2}\left|e^{-a\zeta^2}\right|^2,$$
where $\zeta$ is the same as before and $C_2=|e^d|^2$. It follows that
$$\widetilde{|u|^2}(z)\widetilde{|v|^2}(z)\ge C_1C_2e^{2\alpha|\zeta|^2}=C_1C_2e^{2|b+2az|^2/\alpha}.$$
This shows that $\widetilde{|u|^2}\widetilde{|v|^2}$ is unbounded unless $a=0$.

Therefore, the boundedness of $\widetilde{|u|^2}\widetilde{|v|^2}$ implies that
$$u(z)=e^{bz+c},\qquad v(z)=e^{-bz+d}.$$
By \cite{CPZ}, the product $T_uT_{\overline v}$ is bounded on $F^2_\alpha$. In fact, $T_uT_{\overline v}$ is
a constant times a unitary operator. 

Combining the arguments above and the main result of \cite{CPZ} we have actually proved that the following 
conditions are equivalent for $u$ and $v$ in $F^2_\alpha$:
\begin{enumerate}
\item[(a)] $T_uT_{\overline v}$ is bounded on $F^2_\alpha$.
\item[(b)] $T_uT_{\overline v}$ is a constant multiple of a unitary operator.
\item[(b)] $\widetilde{|u|^2}\widetilde{|v|^2}$ is bounded on $\C$.
\item[(c)] $\widetilde{|u|^2}\widetilde{|v|^2}$ is constant on $\C$.
\end{enumerate}

Recall that in the case of Hardy and Bergman spaces, there is actually an absolute constant $C$ ($4$ for the
Hardy space and $16$ for the Bergman space) such that
$$\widetilde{|u|^2}(z)\widetilde{|v|^2}(z)\le C\|T_uT_{\overline v}\|^2$$
for all $u$, $v$, and $z$. We now show that such an estimate is not possible for the Fock space. To see this,
consider the functions
$$u(z)=e^{\alpha\bar az},\qquad v(z)=e^{-\alpha\bar az}.$$
By calculations done in \cite{CPZ}, we have
$$T_uT_{\overline v}=e^{\alpha|a|^2/2}W_a,$$
where $W_a$ is the Weyl unitary operator defined by $W_af(z)=f(z-a)k_a(z)$. On the other hand, by
calculations done earlier, we have
$$\widetilde{|u|^2}(z)\widetilde{|v|^2}(z)=e^{2\alpha|a|^2}.$$
It is then clear that there is NO constant $C$ such that
$$e^{2\alpha|a|^2}\le Ce^{\alpha|a|^2/2}$$
for all $a\in\C$. Therefore, there is NO constant $C$ such that
$$\sup_{z\in\C}\widetilde{|u|^2}(z)\widetilde{|v|^2}(z)\le C\|T_uT_{\overline v}\|^2$$
for all $u$ and $v$. In other words, the easy direction for Sarason's conjecture in the cases of Hardy
and Bergman spaces becomes difficult for Fock spaces.


\end{document}